\documentclass[11pt]{article}
\usepackage{cite}
\usepackage{mathrsfs}
\usepackage{amsfonts}
\usepackage{amsmath}
\usepackage{amsfonts,amssymb,color}
\usepackage{dsfont}
\usepackage{curves}
\usepackage{mathrsfs}
\usepackage{pifont}
\usepackage{amssymb}
\usepackage{latexsym,amsmath,amssymb,amsfonts,epsfig,graphicx,cite,psfrag}
\usepackage{eepic,color,colordvi,amscd}
\usepackage{hyperref}

\usepackage{enumitem}
\newtheorem{theorem}{Theorem}[section]
\newtheorem{proposition}{Proposition}[section]

\newtheorem{lemma}{Lemma}[section]

\newtheorem{problem}{Problem}[section]
\newtheorem{claim}{Claim}[section]
\newtheorem{conjecture}{Conjecture}[section]

\newcommand{\qed}{\hfill\rule{0.5em}{0.809em}}

\def\emptyset{\mbox{{\rm \O}}}

\newenvironment{proof}{
	\noindent {\it Proof.}\rm}
{\mbox{}\hfill\rule{0.5em}{0.809em}\par}

\def\qed{\hfill \rule{0.5em}{0.809em}}
\begin{document}
	
\title{Some properties of minimally nonperfectly divisible graphs}%Perfect weight divisibility is equivalent to perfect divisibility}
	
	\author{Qiming Hu\footnote{Email: 865550832@qq.com }, \;Baogang  Xu\footnote{Email: baogxu@njnu.edu.cn. Supported by 2024YFA1013902},\; Miaoxia Zhuang\footnote{Corresponding author: 19mxzhuang@alumni.stu.edu.cn }\\\\
		\small Institute of Mathematics, School of Mathematical Sciences\\
		\small Nanjing Normal University, 1 Wenyuan Road,  Nanjing, 210023,  China}
	\date{}
\maketitle

\begin{abstract}
A graph is perfectly divisible if every induced subgraph $H$ admits a partition $(A,B)$ of $V(H)$ such that $H[A]$ is perfect and $\omega(H[B])<\omega(H)$. A graph is minimally nonperfectly divisible if it is not perfectly divisible but every proper induced subgraph is perfectly divisible.
We prove that every $P_2\cup P_4$-free or diamond-free minimally nonperfectly divisible graph has no homogeneous set, and that every $2P_3$-free or claw-free minimally nonperfectly divisible graph
has no clique cutset.
We also study the relationship between perfect divisibility and its weighted analogue, perfect weight divisibility, defined for positive integral vertex weights.
We prove that if a graph $H$ has no adjacent twins, then every $H$-free graph is perfectly divisible if and only if every $H$-free graph is perfectly weight divisible.

	\begin{flushleft}
	\end{flushleft}
	
\end{abstract}

%\newpage

\section{Introduction}
All graphs considered in this paper are finite and simple. Let $G$ be a graph, $v\in V (G)$, and let $X$ and $Y$ be two disjoint subsets of $V(G)$. As usual, we use $N_G(v)$ to denote the set of neighbors of $v$ in $G$, and let $N_G(X)=(\cup_{x\in X}N_G(x))\setminus X$ and $M_G(X)=V(G)\setminus (N_G(X)\cup X)$. We say that $v$ is {\em complete} to $X$ if $X\subseteq N(v)$, and say that $v$ is {\em anticomplete} to $X$ if $X\cap N_G(v)=\emptyset$. We say that $X$ is complete (resp. anticomplete) to $Y$ if each vertex of X is complete (resp. anticomplete) to $Y$. If $1 < |X| < |V (G)|$ and every vertex of $V(G) \setminus X$ is either complete or anticomplete to $X$, then we call $X$ a {\em homogeneous set}. We use $G[X]$ to denote the subgraph of $G$ induced by $X$, and call $X$ a \textit{clique} if $G[X]$ is a complete
graph. The {\em clique number} of $G$, denoted by $\omega(G)$, is the maximum size of cliques in $G$. If it does not cause any confusion, we usually omit the subscript $G$ and simply write $N(v)$, $M(v)$, $N(X)$, $M(X)$. For brevity, we let \(N_X(Y)\) (resp. \(M_X(Y)\)) denote \(N_{G[X]}(Y)\) (resp. \(M_{G[X]}(Y)\)), and write $x\sim y$ (resp. $x\not\sim y$) if $xy\in E(G)$ (resp. $xy\not\in E(G)$). Two adjacent vertices $u$ and $v$ are called {\em adjacent twins} if $N[u]=N[v]$. %For $v\in V(G)$ and $X\subseteq V(G)$, we write $N_X(v)=N(v)\cap X$ and $N_X[v]=N[v]\cap X.$ More generally, for $S\subseteq V(G)$, let $N_X(S)=N_G(S)\cap X$ and $N_X[S]=N_G[S]\cap X.$

We say that $G$ \textit{induces} $H$ if $G$ has an induced subgraph isomorphic to $H$, and say that $G$ is \textit{{\em $H$}-free} otherwise. Analogously, for a family ${\cal H}$ of graphs, we say that $G$ is \textit{${\cal H}$-free} if $G$ induces no member of ${\cal H}$. A clique $X$ of a connected graph $G$ is called a {\em clique cutset} if $G-X$ is disconnected.

For integer $k$, we say that $G$ is $k$-colorable if there is a mapping $\phi$ from $V(G)$ to $\{1,2,...,k\}$ such that $\phi(u)\neq \phi(v)$ whenever $u\sim v$. The chromatic number $\chi(G)$ of $G$ is the smallest integer $k$ such that $G$ is $k$-colorable. If $\chi(H)=\omega(H)$ for each of its induced subgraphs $H$, then $G$ is called a {\em  perfect graph}.

Let $H_1$ and $H_2$ be two graphs on disjoint vertex sets, with
$|V(H_2)|\ge 2$, and let $v\in V(H_1)$. By {\em substituting} $H_2$ {\em for} \(v\), we mean a graph $H$ with vertex set $V(H_1-v)\cup V(H_2)$ and edge set $E(H_1-v)\cup E(H_2)\cup \{xy\;:\; x\in N(v), y\in V(H_2)\}$. It is easy to see that $V(H_2)$ is a homogeneous set of $H$. Lov\'asz \cite{Lovasz} proved the following theorem.
\begin{theorem}\label{thm-1}{\em \cite{Lovasz}}
Let $G$ and $H$ be perfect graphs and $x\in V(G)$. Then the graph obtained from $G$ by substituting $H$ for $x$ is perfect.
\end{theorem}

Ho\`ang \cite{Hoang} introduced the notion of perfect divisibility
and also considered its weighted version; see also
Chudnovsky and Sivaraman \cite{2-divisible2019}.

Let $G$ be a graph, $h$ be a positive integral weight function on $V(G)$ and $X\subseteq V(G)$. We use $\omega_h(G)$ to denote the maximum weight of cliques in $G$ and simply write $\omega_h(G[X])$ as $\omega_h(X)$. Denote $h|_X$ to be a weight function induced on $X$ by $h$. An $h$-{\em perfect division} $(A, B)$ of $G$ is a partition of $V(G)$ into $A$ and $B$ such that $G[A]$ is perfect and $\omega_h(B) < \omega_h(G)$ . A graph $G$ is $h$-{\em perfectly divisible} if for every induced subgraph $H$ of $G$, $H$ has an $h|_{V(H)}$-perfect division. We abbreviate $h$-perfectly divisible simply as {\em perfectly divisible} if $h$ is the all-ones weight function.  A graph $G$ is {\em perfectly weight divisible} if $G$ is $h$-perfectly divisible for every positive integral weight function $h$ on $V(G)$. It is easy to see that if $G$ is perfectly weight divisible, then $G$ is perfectly divisible.% For each vertex $x\in V(G)$, we define a special weight function $h_{x, 2}$ on $G$, by setting $h_{x, 2}(x)=2$ and $h_{x, 2}(v)=1$ for all $v\ne x$. We denote this family of special weight functions by ${\cal H}_2$, i.e., ${\cal H}_2=\{h_{x, 2}|x\in V(G)\}$. A graph $G$ is ${\cal H}_2$-{\em perfectly  divisible} if $G$ is $h$-perfectly divisible for every weight function $h\in{\cal H}_2$.

Let $G$ be a nonperfectly divisible graph. If each of its proper induced subgraphs is perfectly divisible, then we call $G$ a \textit{minimally nonperfectly divisible graph} (MNPD for short). Let $G$ be a nonperfectly weight divisible graph. If each of its proper induced subgraphs is perfectly weight divisible, then we call $G$ a \textit{minimally nonperfectly weight divisible graph} (MNPWD for short).

As a powerful inductive tool, the existence of a homogeneous set plays an important role in studying the perfect divisibility of graphs, see \cite{2-divisible2019,Hoang,DXX2022,SX2024,WX2024,XZ2025--}.
The following weighted result was proved by Chudnovsky and Sivaraman~\cite[Theorem~3.6]{2-divisible2019}, and by Ho\'ang \cite[Lemma 7.2]{Hoang}. 

\begin{lemma}\label{homoset}{\em\cite{2-divisible2019,Hoang}}
	A minimally nonperfectly weight divisible graph has no homogeneous sets.
\end{lemma}

It is natural to ask whether a minimally nonperfectly divisible graph may contain a homogeneous set. It is worth mentioning that by some slight modification to the proof of \cite{WX2024}, one can show that the graphs consider in \cite{WX2024} are perfectly weight divisible.

 Recently, Ho\'ang \cite{Hoang2025} conjectured that an MNPD graph contains no clique cutset, and verified this for $P_5$-free graphs and $4K_1$-free graphs.
\begin{conjecture}\label{coj-MNPD-nocliquecutset}{\em \cite{Hoang2025}}
    A minimally nonperfectly divisible graph contains no clique cutset.
\end{conjecture}

In this paper, we verify Conjecture~\ref{coj-MNPD-nocliquecutset} on graphs that forbid some other small substructure. We also confirm that if we forbid some small substructure, no minimally nonperfectly divisible graph may contain a homogeneous set. 

As usual, we use $P_n$ to denote a path on $n$ vertices. Let $2P_3$ be the disjoint union of two $P_3$'s, $P_2\cup P_4$ the disjoint union of $P_2$ and $P_4$, and let {\em diamond} be the graph obtained from $K_4$ by removing an edge. Our main results are the following three theorems. 

\begin{theorem}\label{thm-3}
     Let $F$ be a $P_2\cup P_4$ or a diamond. Then, no minimally nonperfectly divisible $F$-free graph may contain a homogeneous set.
\end{theorem}

\begin{theorem}\label{thm-4}
    Let $F$ be a $2P_3$ or a claw. Then, no minimally nonperfectly divisible $F$-free graph may contain a clique cutset. 
\end{theorem}

\begin{theorem}\label{thm-5}
Let $H$ be a graph with no adjacent twins. Then the class of $H$-free graphs is perfectly divisible if and only if it is perfectly weight divisible.
\end{theorem}

Theorem~\ref{thm-5} shows that whenever $H$ has no adjacent twins,
perfect divisibility and perfect weight divisibility are equivalent
for the class of $H$-free graphs. 
A natural example is the fork. Sivaraman conjectured that every fork-free graph is perfectly divisible~\cite{KKS2022}.
Since the fork has no adjacent twins, Theorem~\ref{thm-5} shows that
this conjecture is equivalent to the statement that every fork-free
graph is perfectly weight divisible; the latter was recently proved by Liu, Sun, Wang, Wu and Zeng~\cite{Liu2026}. Theorem 1.4 may allow part of their weighted argument to be replaced by an unweighted one.

This paper is organized as follows. In Section \ref{sec1}, we prove Theorem \ref{thm-3}. In Section \ref{sec2}, we discuss clique cutsets in MNPD graphs and prove Theorem \ref{thm-4}. In Section \ref{sec3}, we prove Theorem \ref{thm-5}, discuss some extended results and outline potential directions for future research.

\section{Proof of Theorem \ref{thm-3}}\label{sec1}
We prove Theorem \ref{thm-3} by contradiction, distinguishing the cases where $G$ is $P_2\cup P_4$-free and where $G$ is diamond-free. First, we show a useful lemma.

\begin{lemma}\label{lem-thm3}
Let $G$ be an MNPD graph containing a homogeneous set $X$.
Then $X$ is not stable.
\end{lemma}
\begin{proof}
Suppose to the contrary that $X$ is stable.
Choose $x\in X$, and let $(A,B)$ be a perfect division of
$G[(V(G)\setminus X)\cup\{x\}]$.
For every $S\subseteq V(G)$ containing $x$, we have
\begin{equation}\label{eqa-thm3}
    \omega(S\cup X)=\omega(S).
\end{equation}
Indeed, a clique meets $X$ in at most one vertex.
If it meets $X$, replacing that vertex by $x$ gives a clique
of the same size in $G[S]$, since $X$ is homogeneous.
This proves~(\ref{eqa-thm3}).

If $x\in A$, then $G[A\cup X]$ is perfect by
Theorem~\ref{thm-1}, and $(A\cup X,B)$ is a perfect division of $G$.
If $x\in B$, then
$\omega(B\cup X)=\omega(B)<\omega(G)$ by~(\ref{eqa-thm3}),
and $(A,B\cup X)$ is a perfect division of $G$.
Both cases contradict the choice of $G$.
\end{proof}

\medskip

\noindent\textbf{Proof of Theorem \ref{thm-3}:} Suppose to the contrary that $G$ is an MNPD graph that contains a homogeneous set $X$. By Lemma \ref{lem-thm3}, $X$ is not stable. Suppose that $x_1,x_2$ are two adjacent vertices in $X$. 

\textbf{Case 1:} $G$ is $P_2\cup P_4$-free. Then $G[M(X)]$ is $P_4$-free since $X$ contains an edge, and thus perfect. Let $(A_X,B_X)$ be a perfect division of $G[X]$. The graph $G[A_X\cup M(X)]$ is perfect, since
$G[A_X]$ and $G[M(X)]$ are perfect and
$A_X$ is anticomplete to $M(X)$.
Suppose that $K$ is a clique in $G[B_X\cup N(X)]$
with $|K|=\omega(G)$.
Since $X$ is complete to $N(X)$, replacing $K\cap X$
by a maximum clique of $G[X]$ gives a clique in $G$.
Hence $|K\cap X|=\omega(X)$, contrary to
$K\cap X\subseteq B_X$ and $\omega(B_X)<\omega(X)$.
Thus $\omega(B_X\cup N(X))<\omega(G)$, and
$(A_X\cup M(X),B_X\cup N(X))$ is a perfect division of $G$,
a contradiction.

\textbf{Case 2:} $G$ is diamond-free. Since $G$ is connected and $X$ is a proper nonempty homogeneous set,
$N(X)\ne\emptyset$. Suppose that $\{a,b,c\}\subseteq X$ induces a $P_3$,
and choose $y\in N(X)$. Then $y$ is complete to $\{a,b,c\}$, forcing a diamond on $\{a,b,c,y\}$, a contradiction. Thus, $X$ is $P_3$-free and is the disjoint union of some cliques. If $N(X)$ contains two nonadjacent vertices $y_1$ and $y_2$,
then $\{x_1,x_2,y_1,y_2\}$ induces a diamond, a contradiction.
Thus $N(X)$ is a clique. Let $D$ be the component of $G[X]$ containing $x_1$.
The set $D$ is a clique and is complete to the clique $N(X)$.
Thus $N(x_1)=(D\setminus\{x_1\})\cup N(X)$ is a clique. Let $(A,B)$ be a perfect division of $G-x_1$ , then $(A\cup \{x_1\}, B)$ is a perfect division of $G$ since $N(x_1)$ is a clique, a contradiction.
%, contrary to Theorem~\ref{thm-basinsimplicial set}(1).
This completes the proof of Theorem \ref{thm-3}.\qed

\section{Clique cutset in MNPD graphs}\label{sec2}

In this section, we mainly discuss the structure of MNPD graphs that may contain a clique cutset. Before that, we need to introduce some definitions which are significant for the discussion. 

Let $G$ be a graph, and $X, Y\subseteq V(G)$. Let $N[X]=N(X)\cup X$.  A set $X$ is called a \textit{basin} of $G$ if $\omega(N[X])<\omega(G)$, and is called a \textit{simplicial set} of $Y$ if $N_{X\cup Y}(x)$ is a clique for any $x\in X$. We simply say $X$ is a simplicial set of $G$ when $X$ is a simplicial set of $V(G)$. The proposition below follows directly from the definition.

\begin{proposition}\label{pro-simplicial set}
If $X$ is a simplicial set of $Y$ for some $Y\subset V(G)$ then  $X$  is the union of pairwise anticomplete cliques, and if $X$ is a simplicial set of $G$ then $X$ is a simplicial set of $U$ for any $U\subseteq V(G)$.
\end{proposition}

\begin{proposition}\label{pro-simplicial set-p}
	Let $G$ be a graph and $X$ be a nonempty simplicial set of $G$. Then $G$ is perfect if and only if $G-X$ is perfect.
\end{proposition}
\begin{proof}
    If $G$ is perfect, then all of its induced subgraphs are perfect, and so is $G-X$. 
    
    Suppose that $G-X$ is perfect. Let $Y$ be a subset of $V(G)$. If $Y\subseteq X$ or $Y\cap X=\emptyset$, then $G[Y]$ is clearly perfect by Proposition~\ref{pro-simplicial set}. Suppose that $Y\cap X\ne\emptyset$ and $Y\setminus X\ne \emptyset$. By the definition of simplicial set, $X\cap Y$ is also a simplicial set $Y\setminus X$. Given that no vertex of $X$ is contained in any odd hole or odd antihole (if such structures exist) in $Y$, $G[Y]$ is perfect. Therefore, $G$ is perfect.
\end{proof}

\medskip
Let $X$ be a subset of $V(G)$, and $(X_1, X_2, \ldots, X_k)$ be a partition of $X$. We call $(X_1, X_2, \ldots, X_k)$ a \textit{simplicial decomposition} of $G[X]$ if $X_{i+1}$ is a simplicial set of $\bigcup\limits_{j=1}^{i}X_j$ for $i\in\{1,2,\ldots,k-1\}$. If $(X_1,X_2,\ldots,X_k)$ is a \textit{simplicial decomposition} of $X$ and $G[X_1]$ is perfect, then we call $(X_1,X_2,\ldots,X_k)$ a \textit{perfect simplicial decomposition} of $G[X]$.

\begin{proposition}\label{pro-simplicial set-pd}
	Let $G$ be a graph and $X$ be a nonempty simplicial set of $G$. Then $G$ is perfectly divisible if and only if $G-X$ is perfectly divisible.
\end{proposition}
\begin{proof}
	If $G$ is perfectly divisible, then all of its induced subgraphs are perfectly divisible, and so is $G-X$.  
    
    Suppose that $G-X$ is perfectly divisible. Let $Y$ be a subset of $V(G)$. If $Y\subseteq X$ or $Y\cap X=\emptyset$, then $G[Y]$ is clearly perfectly divisible by Proposition~\ref{pro-simplicial set}. Suppose that $Y\cap X\ne\emptyset$ and $Y\setminus X\ne \emptyset$, and let $A$ and $B$ be a perfect division of $G[Y\setminus X]$ such that $G[A]$ is perfect and $\omega(B)<\omega(Y\setminus X)\le \omega(Y)$. By Proposition~\ref{pro-simplicial set-p}, we have that $G[A\cup(X\cap Y)]$ is perfect, which implies that  $(A\cup(X\cap Y), B)$ is a perfect division of $G[Y]$.  Therefore, $G$ is perfectly divisible.
\end{proof}

\begin{proposition}\label{pro-simplicial set-if}
	Let $G$ be a graph, $X\subseteq V(G)$, and $(X_1,X_2,\ldots,X_k)$ be a simplicial set decomposition of $X$. Then, $G[X]$ is perfect if and only if $G[X_1]$ is perfect.
\end{proposition}
\begin{proof}
If $G[X]$ is perfect, then $G[X_1]$ is certainly perfect. Suppose that $G[X_1]$ is perfect. We will prove by induction that $G[X]$ is perfect. We are done if $k=1$. Suppose that $k=r$ for some $r\ge 2$, and suppose that the statement holds while $k\le r-1$. Then, $G[X_1\cup\cdots\cup X_{r-1}]$ is perfect by induction. By Proposition~\ref{pro-simplicial set-p}, $G[X_1\cup\cdots\cup X_r]$ is perfect since $G[X_1\cup\cdots\cup X_{r-1}]$ is perfect and $X_r$ is a simplicial set of $X_1\cup\cdots\cup X_{r-1}$.
\end{proof}

\begin{theorem}\label{thm-basinsimplicial set}
    Let $G$ be an MNPD graph, and $v\in V(G)$. Then the following hold:
    \begin{itemize}
        \item [$(1)$] $G$ has no nonempty basin or simplicial set.
        \item [$(2)$] $N(v)$ has a subset $Y$ with $\alpha(Y)\ge 2$ and $\omega(N(v)\setminus Y)=\omega(G)-1$.
        \item [$(3)$] There exists a maximum clique $K$ such that $v\in K$ and $N(v)\not\subseteq K$.
        \item [$(4)$] $|N(X)|\ge \omega(G)-\omega(X)$ for any nonempty subset $X$ of $V(G)$. 
        \item [$(5)$] $|N(v)|\ge \omega(G)+1$.
    \end{itemize}
\end{theorem}
\begin{proof} It follows from Proposition \ref{pro-simplicial set-pd} that $G$ has no nonempty simplicial sets. Let $X$ be a nonempty subset of $V(G)$. To prove (1),  it suffices to show that $X$ is not a basin of $G$. Since $G$ is MNPD, we may choose $(A,B)$ to be a perfect division of $G-X$. If $X$ is a basin of $G$, then $\omega(B\cup X)\le\max\{\omega(B), \omega(N[X])\}<\omega(G)$, and so $(A, B\cup X)$ is a perfect division of $G$. Therefore, $X$ is not a basin of $G$. 

Next, we prove (2). Suppose to the contrary that there exists a vertex $v\in V(G)$ such that for any $Y\subseteq N(v)$ with $\alpha(Y)\ge 2$, $\omega(N(v)\setminus Y)<\omega(G)-1$. Since $G$ is MNPD, we may choose $(A, B)$ to be a perfect division of $G-v$. Then $v$ is not a simplicial set of $A$ as otherwise $(A\cup \{v\}, B)$ is a perfect division of $G$. So $\alpha(N(v)\cap A)\ge 2$. Let $W$ be a subset of $N(v)\cap A$ with $\alpha(W)\ge 2$. Then,  $\omega(N(v)\cap B)\le\omega(N(v)\setminus W)<\omega(G)-1$ by our assumption, which implies that $\omega(B\cup \{v\})<\omega(G)$. Therefore, $(A, B\cup \{v\})$ is a perfect division of $G$, a contradiction. %This proves the fourth statement. 

Since $G$ has no basin or simplicial set by (1), we have that $\omega(N[v])=\omega(G)$ and $N(v)$ is not a clique, and so (3) holds.

Let $X$ be a nonempty subset of $V(G)$. Since $X$ is not a basin of $G$ by (1), we have that $\omega(G)=\omega(N[X])\le \omega(X)+\omega(N(X))$. It follows that $|N(X)|\ge \omega(N(X))\ge \omega(G)-\omega(X)$. 

The last statement follows directly from (2).
\end{proof}

\medskip

It is worth mentioning that Theorem~\ref{thm-basinsimplicial set}(5) cannot be improved further, as evidenced by the Gr\"otzsch graph. 

Now, we come back to MNPD graphs with a clique cutset. Let $G$ be an MNPD graph and $X$ be a clique cutset of $G$. Then $V(G)\setminus X$ can be partitioned into two nonempty sets $V_1$ and $V_2$ such that $V_1$ is anticomplete to $V_2$. Let $G_i = G[X\cup V_i]$ for $i = 1, 2$. By the minimality of $G$, for $i = 1, 2$, $G_i$ admits a perfect division. By Theorem \ref{thm-basinsimplicial set}(1), $\omega(G_1)=\omega(G_2)=\omega(G)$ since $G$ has no nonempty basin. We can see that $G$ is $\binom{\omega(G)+1}{2}$-colorable since $G_1$ and $G_2$ are $\binom{\omega(G)+1}{2}$-colorable. Since all graphs with $\chi(G)\le \omega(G)+1$ have a perfect division, it follows that
 \begin{lemma}
     Each MNPD triangle-free graph contains no clique cutset.
 \end{lemma}

  Let $G$ be an MNPD graph with a clique cutset $X$, and $G_1$ and $G_2$ be defined as above. Then, we have the following Lemmas~\ref{lem-MNPD-notperfect}, \ref{lem-MNPD-notempty} and \ref{lem-MNPD-x}.

\begin{lemma}\label{lem-MNPD-notperfect}{\em\cite[Lemma~2.3]{Hoang2025}}
  For $i\in\{1,2\}$, $G_i$ is not perfect.
\end{lemma}

 \begin{lemma}\label{lem-MNPD-notempty}
 $X\setminus ((A_1\cap A_2)\cup (B_1\cap B_2))\ne\emptyset$ for any perfect division $(A_i, B_i)$ of $G_i$, $i\in\{1,2\}$.
 \end{lemma}
\begin{proof}
     If $X\setminus ((A_1\cap A_2)\cup (B_1\cap B_2))=\emptyset$, then $(A_1\cup A_2, B_1\cup B_2)$ is a partition of $V(G)$. Additionally, $A_1\cap A_2$ and $B_1\cap B_2$ are cliques since $X$ is a clique cutset, which implies that $G[A_1\cup A_2]$ is perfect and $\omega(B_1\cup B_2)< \omega(G)$. Thus, $(A_1\cup A_2, B_1\cup B_2)$ forms a perfect division of $G$, a contradiction.     
\end{proof}

\begin{lemma}\label{lem-MNPD-x}
    There exists a perfect division $(A_i, B_i)$ of $G_i$ for each $i\in\{1,2\}$ such that 
    \begin{itemize}
        \item for any $x\in X\cap A_1\cap B_2$, $\omega(B_1\cup \{x\})=\omega(G_1)$ and $G[A_2\cup \{x\}]$ is not perfect, and
        \item for any $x\in X\cap A_2\cap B_1$, $\omega(B_2\cup \{x\})=\omega(G_2)$ and $G[A_1\cup \{x\}]$ is not perfect.
    \end{itemize}
\end{lemma}
\begin{proof}
Suppose for contradiction that for any perfect division $(A_i, B_i)$ of $G_i$ where $i\in\{1,2\}$, there exists a vertex $x\in X\cap A_1\cap B_2$, such that $\omega(B_1\cup \{x\})<\omega(G_1)$ or $G[A_2\cup \{x\}]$ is perfect, or there exists a vertex $y\in X\cap A_2\cap B_1$, such that $\omega(B_2\cup \{y\})<\omega(G_2)$ or $G[A_1\cup \{y\}]$ is perfect. 

Let $(A_i^0, B_i^0)$ be a perfect division of $G_i$ for $i=1,2$. By assumption and by symmetry, suppose $x^0\in X\cap A_1^0\cap B_2^0$ such that $\omega(B_1^0\cup \{x^0\})<\omega(G_1)$ or $G[A_2^0\cup \{x^0\}]$ is perfect. Then $(A_1^0\setminus \{x^0\}, B_1^0\cup \{x^0\})$ is a new perfect division of $G_1$ or $(A_2^0\cup \{x^0\}, B_2^0\setminus \{x^0\})$ is a new perfect division of $G_2$. In either case, let $(A_i^1, B_i^1)$ denote the resulting perfect division of $G_i$ for $i=1,2$. Note that $x^0\in (A_1^1\cap A_2^1)\cup (B_1^1\cap B_2^1)$, so $|X\setminus ((A_1^1\cap A_2^1)\cup (B_1^1\cap B_2^1))|<|X\setminus ((A_1^0\cap A_2^0)\cup (B_1^0\cap B_2^0))|$. Repeating this process eventually yields a perfect division $(A_i^j, B_i^j)$ of $G_i$, for some $j$, such that $X\setminus ((A_1^j\cap A_2^j)\cup (B_1^j\cap B_2^j))=\emptyset$, a contradiction to Lemma~\ref{lem-MNPD-notempty}.
\end{proof}

\medskip

Let  $(A_i, B_i)$ denote the perfect division of $G_i$ that satisfies Lemma \ref{lem-MNPD-x}, with $x\in A_1\cap B_2\cap X$. We will show that 
    \begin{equation}\label{eqa-P4}
        \mbox{there exists $z\in( A_2\cap X)\cup \{x\}$ such that $G[(A_2\setminus X)\cup \{z\}]$ contains a $P_4$ starting at $z$.}
    \end{equation}

Since $G[A_2\cup \{x\}]$ is not perfect by Lemma \ref{lem-MNPD-x}, let $C$ be an odd hole or an odd antihole of $G[A_2\cup \{x\}]$. If $C$ is an odd hole, let $C=xv_1v_2...v_nx$ with $n\ge 4$ even, then since $X$ is a clique, we have that $X\cap V(C)\subseteq \{x,v_1\}$ or $\{x,v_n\}$, and so $xv_1v_2v_3$ or $xv_nv_{n-1}v_{n-2}$ is the desired path with $z=x$. Suppose that $C$ is an odd antihole, and 
write $V(C)=\{u_0,u_1,u_2,...,u_k\}$ where $k\ge 6$ even, such that $u_0=x$ and the complement of $C$ is just the cycle $u_0u_1u_2\ldots u_ku_0$. There must exist a vertex $u_j\in V(C)\cap X$ such that $|\{u_{j+2},u_{j-2}\}\cap X|\le 1$; otherwise $|X\cap V(C)|\ge \frac{k+1}{2}$, a contradiction to $\omega(C)=\frac{k}{2}$. Since $\{u_{j+1},u_{j-1}\}\cap X=\emptyset$, either $u_ju_{j+2}u_{j-1}u_{j+1}$ or $u_ju_{j-2}u_{j+1}u_{j-1}$ is the desired path with $z=u_j$. This proves (\ref{eqa-P4}). 

\begin{lemma}\label{L-2}
If $G$ is an MNPD claw-free graph, then $G[M(v)]$ contains no odd antihole except $C_5$ for any $v\in V(G)$.
\end{lemma}
\begin{proof}
    Suppose to the contrary that $G$ is an MNPD claw-free graph, $v\in V(G)$ and $C$ is an odd antihole of $G[M(v)]$ such that $|V(C)|\ge 7$. Let $V(C)=\{v_0,v_1,v_2,...,v_k\}$ for some even $k\ge 6$, such that the complement of $C$ is a cycle $v_0v_1v_2\ldots v_kv_0$.
    Let $u\in N(M(C))$, $a$ be a neighbor of $u$ in $M(C)$, and suppose, without loss of generality, that $u\sim v_0$. To avoid claws on $\{u,a,v_0,v_1\}$ and
$\{u,a,v_0,v_k\}$, we have
$u\not\sim v_1$ and $u\not\sim v_k$.
    If $u\not\sim v_2$, then $u\sim v_3$ to avoid a claw on $\{v_0,v_2,v_3,u\}$ and $u\not\sim v_4$ to avoid a claw on $\{a,v_3,v_4,u\}$, and consequently, $u\sim v_5$ to avoid a claw on $\{v_0,v_4,v_5,u\}$, which forces a claw on $\{u,v_5,v_1,v_2\}$, a contradiction. Thus $u\sim v_2$, then $u\not\sim v_3$ to avoid a claw on $\{a,u,v_2,v_3\}$, and $u\sim v_4$ to avoid a claw on $\{u,v_0,v_3,v_4\}$, and consequently, $u\sim v_i$ for $i\in\{0,2,...,k-2\}$ and $u\not\sim v_j$ for $j\in\{1,3,...,k-1\}$. But now, $\{v_2,v_{k-1},v_{k},u\}$ induces a claw, a contradiction again. This completes the proof of Lemma \ref{L-2}.
\end{proof}

\medskip

\noindent\textbf{Proof of Theorem \ref{thm-4}:} 
    Suppose to the contrary that $G$ is an MNPD graph that contains a clique cutset, and let $X$ be a minimum clique cutset. Then $V(G)\setminus X$ can be partitioned into two nonempty sets $V_1$, $V_2$ such that $V_1$ is anticomplete to $V_2$. Let $G_i = G[X\cup V_i]$ for $i = 1, 2$. By Lemmas~\ref{lem-MNPD-notempty} and \ref{lem-MNPD-x}, we may choose $(A_i, B_i)$ as a perfect division of $G_i$ for $i\in\{1,2\}$ and assume by symmetry that 
\begin{equation}\label{eqa-thm-2-cliquecut}
\mbox{$x_0\in X\cap A_1\cap B_2$ satisfies $\omega(B_1\cup \{x_0\})=\omega(G_1)$ and $G[A_2\cup \{x_0\}]$ is not perfect}.
\end{equation}
    
     By (\ref{eqa-P4}), $G[A_2\setminus X]$ induces a $P_3$, say $w_1w_2w_3$. If $G$ is $2P_3$-free, then to avoid a $2P_3$ on  $\{w_1,w_2,w_3,x_1,x_2,x_3\}$ for some $x_1,x_2,x_3\in V_1$, we have that each component of $G[V_1]$ is a clique, which implies that $G_1$ is perfect, a contradiction to Lemma \ref{lem-MNPD-notperfect}. 
     
  \medskip 
  
  Next, we suppose that $G$ is claw-free.  Let $x\in X$. To avoid a claw on $\{x,y_1,y_1',y_2\}$ or $\{x,y_1,y_2',y_2\}$ for some $y_1,y_1'\in V_1$ and $y_2,y_2'\in V_2$, we have that
   \begin{equation}\label{eqa-MNPD-forksimplicial set}
    N_{V_i}(x)\text{ is a clique for every $x\in X$ and $i\in\{1,2\}$.}
   \end{equation}

   By the Strong Perfect Graph Theorem \cite{SPGT}, since $G[A_2]$ is perfect but $G[A_2\cup\{x_0\}]$ is not by (\ref{eqa-thm-2-cliquecut}), we may suppose that $G[A_2\cup\{x_0\}]$ induces an odd hole or an odd antihole, denoted by $C$ with $x_0\in V(C)$. 

\begin{claim}\label{cla-1}
    $C$ is not an odd hole.
\end{claim}   
\begin{proof}
   Suppose that $C$ is an odd hole, write $C=x_0v_1v_2\ldots v_nx_0$ with $n\ge 4$ even. Since $X$ is a clique, we may suppose $V(C)\cap X=\{x_0,v_1\}$ by (\ref{eqa-MNPD-forksimplicial set}). To avoid a claw on $\{x_0,v_1,v_n,y_1\}$ for some $y_1\in V_1$, $N_{V_1}(x_0)$ must be complete to $\{x_0,v_1\}$. By symmetry, $N_{V_1}(v_1)$ is also complete to $\{x_0,v_1\}$. Thus, $N_{V_1}[\{x_0,v_1\}]$ is a clique by (\ref{eqa-MNPD-forksimplicial set}). Let $x^*\in X\setminus\{x_0,v_1\}$. We next prove that 
   \begin{equation}\label{eqa-MNPD-2}
       \mbox{$N_{V_1}[x^*]= N_{V_1}[\{x_0,v_1\}]$.}
   \end{equation}
   Firstly, we prove that $N_{V_1}[x^*]\subseteq N_{V_1}[\{x_0,v_1\}]$. Suppose for contradiction that there exists a $y_1\in N_{V_1}[x^*]\setminus N_{V_1}[\{x_0,v_1\}]$. Since $X$ is a minimum clique cutset, we may choose a  $y_2\in N_{V_2}(x^*)$. To avoid a claw on $\{x_0,y_1,y_2,x^*\}$ or $\{v_1,y_1,y_2,x^*\}$, $y_2$ must be complete to $\{x_0,v_1\}$ since $y_1\not\sim x_0$ and $y_1\not\sim v_1$ by assumption. Let $z\in N_{V_1}(\{x_0,v_1\})$. To avoid a claw on $\{x_0,y_2,v_n,z\}$ or $\{v_1,y_2,v_2,z\}$, we have $y_2\sim v_n$ and $y_2\sim v_2$. Then to avoid a claw on $\{y_2,x^*,v_2,v_n\}$, we have that either $x^*\sim v_n$ or $x^*\sim v_2$, which forces a claw on $\{y_1,x^*,v_n,v_1\}$ or $\{y_1,x^*,v_2,x_0\}$, a contradiction. Therefore, $N_{V_1}[x^*]\subseteq N_{V_1}[\{x_0,v_1\}]$.
   
   Next, let $z'\in N_{V_1}[\{x_0,v_1\}]\setminus N_{V_1}[x^*]$. It is certain that $z'$ is complete to $N_{V_1}[\{x_0,v_1\}]\setminus\{z'\}$ since $N_{V_1}[\{x_0,v_1\}]$ is a clique. To avoid a claw on $\{x_0,z',v_n,x^*\}$ or $\{v_1,z',v_2,x^*\}$, $x^*\sim v_n$ and $x^*\sim v_2$, which forces a claw on $\{x^*,z,v_2,v_n\}$  for some $z\in N_{V_1}(\{x_0,v_1\})$, a contradiction. This proves (\ref{eqa-MNPD-2}).

   Since $X$ is a clique, we have that for any $x^*\in X$, $N_{V_1\cup X}[x^*]= N_{V_1\cup X}[\{x_0,v_1\}]$, which is a clique by (\ref{eqa-MNPD-forksimplicial set}). So, 
   \begin{equation}\label{eqa-MNPD-3}
       \mbox{$X$ is a simplicial set of $V_1\cup X$, and $N_{V_1}[X]$ is a clique.}
   \end{equation}
   
   Let $K\subseteq B_1\cup\{x_0\}$ be a maximum clique of $G_1$. Then $K=N_{G_1}[X]$ by (\ref{eqa-MNPD-3}). By our assumption, $\omega(B_1\cup\{x_0\})=\omega(G_1)=\omega(G)$, so $X\setminus \{x_0\}\subseteq B_1$. Now $X\cap A_1=\{x_0\}$, $X\cap B_1=X\setminus \{x_0\}$, which implies that $v_1\in A_2\cap B_1\cap X$. By Lemma \ref{lem-MNPD-x}, $G[A_1\cup \{v_1\}]$ is not perfect, a contradiction to the fact that $\{v_1\}$ is a simplicial set of $V_1\cup X$ by (\ref{eqa-MNPD-3}). This completes the proof of Claim~\ref{cla-1}.
\end{proof}

  By Claim~\ref{cla-1}, $C$ is an odd antihole of length at least 7. Let $V(C)=\{u_0,u_1,u_2,...,u_k\}$ for some even $k\ge 6$, such that $u_0=x_0$ and the complement of $C$ is the cycle $u_0u_1u_2\ldots u_ku_0$.

  By (\ref{eqa-MNPD-forksimplicial set}), $N_{A_2}(x_0)$ can be partitioned into two cliques $X_1$ and $X_2$, where $X_1\subseteq X\cap A_2$ and $X_2\subseteq A_2\setminus X$. Reversing the cyclic order if necessary, we may assume that
$\{u_2,u_4,\ldots,u_{k-2}\}\subseteq X_1$ and
$\{u_3,u_5,\ldots,u_{k-1}\}\subseteq X_2$.
  It is clear that $\{u_1,u_k\}\subseteq A_2\setminus X$. We will prove that
  \begin{equation}\label{eqa-MNPD-4}
      \mbox{$V_1$ is complete to $\{u_2,u_4,...,u_{k-2}\}$.}
  \end{equation}
  Let $y_1\in V_1$.  By Lemma \ref{L-2}, every $v\in V(G)\setminus V(C)$ has a neighbor in $C$. If $N_C(y_1)=\{u_0\}$, then $\{u_0,y_1,u_2,u_3\}$ induces a claw, a contradiction. We may suppose that $z\in N_C(y_1)\setminus \{u_0\}$. Then, $z=u_j$ for some $j\in\{2,4,...,k-2\}$. To avoid a claw on $\{y_1,u_j,u_{j+2},u_{j+3}\}$
when $j\le k-4$, we have $y_1\sim u_{j+2}$.
Similarly, avoiding a claw on
$\{y_1,u_j,u_{j-2},u_{j-3}\}$ when $j\ge4$
gives $y_1\sim u_{j-2}$. %To avoid a claw on $\{y_1, u_j,u_{j+2},u_{j+3}\}$ or $\{y_1, u_j,u_{j-2},u_{j-3}\}$,  we have that $y_1$ must be complete to $\{u_{j-2}, u_{j+2}\}$. 
Repeating this argument, we conclude that $y_1$ is complete to $\{u_2,u_4,...,u_{k-2}\}$. This proves (\ref{eqa-MNPD-4}).

  By (\ref{eqa-MNPD-forksimplicial set}), (\ref{eqa-MNPD-4}), and the fact that $X$ is a clique, $V_1\cup X$ can be partitioned into two cliques. Thus, $G_1$ is perfect, which contradicts Lemma \ref{lem-MNPD-notperfect}. This completes the proof of Theorem \ref{thm-4}.\qed%\end{proof}

\section{Proof of Theorem \ref{thm-5}}\label{sec3}
In this section, we prove Theorem~\ref{thm-5} and discuss related questions concerning perfect divisibility. In the first version of this paper, we asked the following two problems concerning perfect divisibility.

\begin{problem}\label{coj-pd-vAB}
    Let $G$ be a perfectly divisible graph, and $v\in V(G)$. Does there exist a perfect division $(A,B)$ of $G$ such that $v\in A$?
\end{problem}

%We end this paper with the following problem.
\begin{problem}\label{pro-2}
    Is there a perfectly divisible graph which is not perfectly weight divisible?
\end{problem}

After the first version of this paper was posted on arXiv,
Chen~\cite[Corollary~1.3]{Chen2026} gave a negative answer to
Problem~\ref{coj-pd-vAB}.
More precisely, for every $t\ge 1$, Chen constructed a perfectly divisible
graph $H_t$ containing two adjacent vertices $b$ and $c$ such that
$b,c\in B$ for every perfect division $(A,B)$ of $H_t$.
His construction also yields MNPD graphs with cutvertices, showing that Conjecture~\ref{coj-MNPD-nocliquecutset} does not hold in general. 

The graphs $H_t$ also give an affirmative answer to
Problem~\ref{pro-2}.
Indeed, $H_t$ admits no $h$-perfect division when $h(b)=2$
and $h(x)=1$ for every $x\in V(H_t)\setminus\{b\}$.
Replacing $b$ in $H_t$ by two adjacent twins yields
an MNPD graph containing a homogeneous set.
%The same construction also gives an affirmative answer to the last problem: the graph $G_t=H_t-v_t$ is perfectly divisible but admits no $h$-perfect division when $h(b)=2$ and $h(x)=1$ for every $x\ne b$. Replacing $b$ in $G_t$ by two adjacent twins yields an MNPD graph containing a homogeneous set.

Chen's result motivates us to further investigate the relationship between perfect divisibility and perfect weight divisibility. Since these two properties are not equivalent in general, it is natural to ask for which graph classes they are equivalent.

%These examples show that perfect divisibility does not imply perfect weight divisibility in general. We next show that if a graph class is closed under clique substitutions and every graph in the class is perfectly divisible, then every graph in the class is perfectly weight divisible.

We first give a useful observation that allows us to replace weights
by clique substitutions. Let $G$ be a graph, let $h$ be a positive integral weight function on $V(G)$, and let $x\in V(G)$. Let \(G_x\) be obtained from \(G\) by substituting a clique \(X\) of size \(h(x)\) for \(x\). Define a weight function $h_x$ on $V(G_x)$ by
\[
h_x(v)=
\begin{cases}
h(v),&v\notin X,\\
1,&v\in X.
\end{cases}
\]
\begin{lemma}\label{lem-1}
If $G_x$ is $h_x$-perfectly divisible, then $G$ is $h$-perfectly divisible.
\end{lemma}
\begin{proof}
Let $F$ be an induced subgraph of $G$. If $x\notin V(F)$, then $F$ is also an induced subgraph of $G_x$ and $h_x|_{V(F)}=h|_{V(F)}$, so $F$ has an $h|_{V(F)}$-perfect division.

Suppose that $x\in V(F)$. Let
\(F_x=G_x[(V(F)\setminus\{x\})\cup X].
\)
Since $G_x$ is $h_x$-perfectly divisible, $F_x$ has an $h_x$-perfect division $(A,B)$. Moreover,
\(
\omega_{h_x}(F_x)=\omega_h(F).
\)

First suppose that $A\cap X\ne\emptyset$. Choose $x_A\in A\cap X$, and identify the original vertex $x$ with $x_A$. Set
\[
A'=(A\setminus X)\cup\{x\},\qquad B'=B\setminus X.
\]
Then $G[A']$ is isomorphic to $G_x[(A\setminus X)\cup\{x_A\}]$ and is therefore perfect. Also,
\[
\omega_h(B')=\omega_{h_x}(B')\le \omega_{h_x}(B)<\omega_{h_x}(F_x)=\omega_h(F).
\]
Thus $(A',B')$ is an $h|_{V(F)}$-perfect division of $F$.

Now suppose that $X\subseteq B$. Set
\[
A'=A,\qquad B'=(B\setminus X)\cup\{x\}.
\]
The graph $G[A']$ is perfect. Since any two vertices of $X$ are adjacent twins and $X\subseteq B$, every clique in $G_x[B]$ meeting $X$ can be enlarged to contain all of $X$. Consequently,
\[
\omega_h(B')=\omega_{h_x}(B)<\omega_{h_x}(F_x)=\omega_h(F).
\]
Hence $(A',B')$ is again an $h|_{V(F)}$-perfect division of $F$. Therefore $G$ is $h$-perfectly divisible.
\end{proof}

\begin{proposition}\label{pro-H-free-substitution}
Let $H$ be a graph.
The class of $H$-free graphs is closed under clique substitutions
if and only if $H$ has no adjacent twins.
\end{proposition}
\begin{proof}
    First suppose that $H$ has no adjacent twins. Let $\mathcal G$ be the class of $H$-free graphs. Let $G\in\mathcal G$, and let $G_x$ be obtained from $G$ by substituting a clique $X$ for $x$. Suppose that $G_x$ contains an induced copy $H'$ of $H$. Since any two vertices of $X$ are adjacent twins in $G_x$, and $H$ has no adjacent twins, $H'$ contains at most one vertex of $X$. If $V(H')\cap X=\emptyset$, then $G$ contains an induced copy of $H$. If $V(H')\cap X=\{x'\}$, then replacing $x'$ with the original vertex $x$ yields an induced copy of $H$ in $G$. Both cases are contradictions. Hence $G_x$ is $H$-free.

    Conversely, suppose that $H$ contains two adjacent twins $u$ and $v$.
    Let $G=H-v$. Clearly, $G$ is $H$-free.
    If we substitute a $2$-clique for $u$ in $G$, the resulting graph
    is isomorphic to $H$, since $u$ and $v$ are adjacent twins in $H$.
    Therefore the class of $H$-free graphs is not closed under clique
    substitutions.
\end{proof}

\begin{proposition}\label{pro-class-pd-pwd}
Let $\mathcal G$ be a class of graphs that is closed under clique
substitutions.
Then $\mathcal G$ is perfectly divisible if and only if
$\mathcal G$ is perfectly weight divisible.
\end{proposition}
\begin{proof}
It remains to prove the nontrivial direction. Suppose that every graph in $\mathcal G$ is perfectly divisible. Let $G\in\mathcal G$, and let $h$ be a positive integral weight function
on $V(G)$. Let $S=\{x\in V(G):h(x)\ne 1\}=\{x_1,...,x_s\}$. 
Put $(G_0,h_0)=(G,h)$.
For each $i\in\{1,\ldots,s\}$, obtain $G_i$ from $G_{i-1}$
by substituting a clique of size $h(x_i)$ for $x_i$, and define
$h_i$ as in Lemma~\ref{lem-1}.

Since $\mathcal G$ is closed under clique substitutions, $G_i\in\mathcal G$ for every $i\in\{1,\ldots,s\}$.
In particular, $G_s$ is perfectly divisible.
Since $h_s$ is the all-ones weight function, $G_s$ is
$h_s$-perfectly divisible. Applying Lemma~\ref{lem-1} successively for
$i=s,s-1,\ldots,1$, we conclude that $G_{i-1}$ is
$h_{i-1}$-perfectly divisible for every $i$.
In particular, $G$ is $h$-perfectly divisible.
Since $h$ was arbitrary, $G$ is perfectly weight divisible.  
\end{proof}

\medskip

\noindent\textbf{Proof of Theorem \ref{thm-5}:} 
Let $\mathcal G$ be the class of all $H$-free graphs.
Since $H$ has no adjacent twins,
Proposition~\ref{pro-H-free-substitution} implies that
$\mathcal G$ is closed under clique substitutions.
The result now follows from Proposition~\ref{pro-class-pd-pwd}.\qed

\medskip

When $H$ has adjacent twins, Proposition~\ref{pro-H-free-substitution} shows that the class of $H$-free graphs is not closed under clique substitutions.
Thus the preceding argument does not apply directly.
In particular, $2K_2$ has adjacent twins, which leads to the following question.
\begin{problem}
Is every $2K_2$-free graph perfectly weight divisible?
\end{problem}

\end{document}